\documentclass{article}

\usepackage{latexsym}
\usepackage{amsmath}
\usepackage{amsthm} 
\usepackage{amstext}
\usepackage{amsfonts} 
\usepackage{amssymb} 
\usepackage{graphics}
\usepackage{graphicx}

\newtheorem{thrm}{Theorem}[section]
\newtheorem{lemma}[thrm]{Lemma}
\newtheorem{cor}[thrm]{Corollary}

\newtheorem{prop}[thrm]{Proposition}

\theoremstyle{definition}
\newtheorem{defn}[thrm]{Definition}

\begin{document}

\title{\bf Some complementary Gray codes}

\author{\Large Adam Hoyt \\ 
18-7825 Avenue Mountain Sights \\
Montr\'{e}al QC H4P 2B1, Canada\\
{\tt ahoyt@connect.carleton.ca} \\
\and
\Large Brett Stevens\footnote{Supported by NSERC, CFI, OIT and the Ontario MRI.} \\ School of Mathematics and Statistics\\ Carleton University \\ 
1125 Colonel By Drive \\ 
Ottawa ON K1S 5B6, Canada \\
{\tt brett@math.carleton.ca}}

\date{}

\maketitle

\begin{abstract}
A {\em complementary} Gray code for binary $n$-tuples is one that, when all the tuples are complemented, is identical to itself; this is equivalent to the complement of the first half of the code being identical to the second half.  We generalize the notion of complementary to $q$-ary $n$-tuples, fixed size combinations of an $n$-set and permutations and, in each case, construct complementary Gray codes.  We relax, as weakly as possible, the notions of complementary to cases where necessary conditions for existence are violated and construct Gray codes within the weakened definitions: these include binary $n$-tuples when $n$ is odd and Lee metric $q$-ary $n$-tuples when $n$ is odd and $q$ is even.  Finally a lemma used in the construction for permutations offers the first known cyclic Gray code for the permutations of a particular family of multisets.

\end{abstract}

\section{Introduction} \label{intro}
A {\em Gray code} for binary $n$-tuples is a linear ordering of the tuples so that consecutive tuples differ in exactly one bit position.  It is {\em cyclic} if the first and last tuples also differ in only one bit position.  Binary $n$-tuples can be associated with subsets of an $n$-set by considering the tuple to be the incidence vector of the subset.  In this context, the {\em complement} of a tuple, $w$, is the incidence vector of the complement of the subset corresponding to $w$; it can be directly constructed by swapping the positions of all 0s and 1s.  We denote the complement of $w$ by $\overline{w}$. 

A ``complementary'' Gray code is a Gray code where the complement, $\overline{w}$, of any word, $w$, occurs exactly half way through the code after $w$ appears.  The existence of complementary Gray codes for binary $n$-tuples for $n$ even is well known, although we are unaware of its original publication \cite{MR2251595}.  When $n$ is even we construct such a complementary Gray code as follows: let $P$ be any Hamilton path from $00\cdots0$ to $11\cdots1$ in the binary $(n-1)$-dimensional hypercube.  Then \label{page:n_even}
\[
0P,1\overline{P}
\]
gives a complementary Gray code.  We use $\overline{P}$ to denote the complement of the words in $P$ in the same order \cite{MR2251595}.  Similarly, when we use $P$ with any operation we mean to apply that operation to each element of $P$ and leave the resulting elements in the same order as from $P$.  For example. $P + x$ indicates adding $x$ to every element of $P$ and keeping the resulting objects in the same order as $P$.  $0P$ means prefixing every element of $P$ with a 0.

There are several ways to make the necessary ingredient, the required Hamilton path, for this construction.  For instance, Frank Ruskey suggests the following construction:
\begin{eqnarray*}
E(n)&:&0^n  \rightarrow 1^n \\
O(n)&:&0^n \rightarrow 01^{n-1} \\
E(n+1) &=& O(n)0, BRRC(n)1 \\
O(n+1) &=& E(n)0,BRC(n)1
\end{eqnarray*}
where $BRC(n)$ is the complement of Gray's original binary reflected Gray code and $BRRC(n)$ is the reversal of the complement of the original binary reflected Gray code \cite{ruskey:pc:2008}. The monotone Gray codes constructed by Savage and Winkler are also suitable ingredients for the construction \cite{MR1329390}.  For example, Table~\ref{comp1} shows a complementary binary Gray code with $4$ bits.
\begin{table}[ht]
\begin{center}
\begin{tabular}{llllllll}
0000 & 0001 & 0011 & 0010 & 0110 & 0100 & 0101 & 0111 \\
1111 & 1110 & 1100 & 1101 & 1001 & 1011 & 1010 & 1000
\end{tabular}
\end{center}
\caption{ A complementary binary Gray code with 4 bits. \label{comp1}}
\end{table}

We call the parity of the number of $1$s in a binary $n$-tuple, $w$, the {\em parity} of $w$.  Since any binary Gray code changes precisely one bit between every pair of consecutive words, the parity of the words alternates between even and odd number.  When $n \geq 2$, the pair of words at distance $2^{n-1}$ in the code must have the same parity since the distance between them is even. However, when $n$ is odd, complementary words have opposite parity.  Thus a complementary $n$-bit binary Gray code cannot exist when $n$ is odd.  In this paper we will suitably define and construct complementary Gray codes for binary $n$-tuples for $n$ odd, $q$-ary $n$-tuples, combinations of an $n$-set and permutations of $n$-objects.  When the existence of a code with a strict definition of ``complementary'' is forbidden by necessary conditions, we will relax the definition as little as possible to construct a code.

We start by introducing some definitions and notations that we will use in this article.  For a word, $w \in \mathbb{Z}_q^n$, we denote the set of non-zero indices as the {\em support} of $w$ and the cardinality of this set is the word's {\em weight}.  When $q=2$ there is a bijective correspondence between words of length $n$, $w \in \mathbb{Z}_2^n$ and the subsets $S_w \subseteq \{1,2,\cdots,n\}$.  The bijection map is simply the support of the word.  For example, $001101$ corresponds to the set $\{3,4,6\}$.  There are two natural metrics that can be placed on $\mathbb{Z}_q^n$.  The {\em Hamming distance} between $v,w \in \mathbb{Z}_q^n$ is the number of entries where $v$ and $w$ differ, or equivalently, the {\em weight} of $v-w$.  Let $v=v_1,v_2,\cdots, v_n$ and $w=w_1,w_2,\cdots, w_n$.  The {\em Lee distance} in the $i$th coordinate is the size of the difference between $v_i$ and $w_i$ in $\mathbb{Z}_q$ regarded as a cycle; in other words $\min\{|v_i-w_i|,q-|v_i-w_i|\}$. The {\em Lee distance} between words is then sum of the Lee distances over all coordinates \cite{MR1664228}:
\[
\sum_{i=1}^{n} \min\{|v_i-w_i|,q-|v_i-w_i|\}.
\]

In graph theory terminology, Hamming distances correspond to distances in the graph 
\[
K_q \Box K_q \Box \cdots \Box K_q
\]
and Lee distances correspond to distances in the multi-dimensional torus
\[
C_q \Box C_q \Box \cdots \Box C_q,
\]
where vertices are labelled by length $n$ $q$-ary words.  In either case, a (cyclic) Gray code with consecutive words at distance one is equivalent to a Hamilton path (cycle) in the corresponding graph.

\section{Gray codes of $q$-ary $n$-tuples} \label{sec:qn}

We first address the problem of defining and constructing complementary Gray codes for binary $n$-tuples when $n$ is odd.

\subsection{Binary $n$-tuples for $n$ odd} \label{subsec:binary}

When $n$ is odd, there are two possible ways to relax the notion of ``complementary'': we can drop the requirement that all words appear in the complementary code or we can keep all words, but allow complementary words to be almost, rather than exactly, halfway through the code.  We show that one can construct Gray codes for all odd $n$ in either case and do so as optimally as possible: leaving out the fewest words or having complementary words be exactly one position off from halfway.

We start with a result that will be used in the construction of these Gray codes.
\begin{lemma} \label{rev_ref}
Let $G$ be a binary Gray code on $n$ bits with $G(i)$ denoting the $i$th word (taken modulo $2^n$).  Then
\[
G'(i) = \left \{ \begin{array}{ll} G(\lfloor i/2 \rfloor)0 & \mbox{if $i \equiv 0,3 \bmod 4$}  \\
G(\lfloor i/2 \rfloor)1 & \mbox{if $i \equiv 1,2 \bmod 4$} \end{array} \right .
\]
defines a binary Gray code on $n+1$ bits. $G'$ is cyclic if and only if $G$ is cyclic.
\end{lemma}
Applying this lemma repeatedly to the 1 bit binary cyclic Gray code, \[0,1,\] yields an alternative construction of Gray's original binary reflected Gray code. We can use Lemma~\ref{rev_ref} to construct a complementary Gray code missing only two words.
\begin{thrm}
  If $n$ is odd there exists a Gray ordering of
  \[
    \mathbb{Z}_2^n\setminus\{00\cdots0,11\cdots1\}
  \]
  such that 
the distance between any pair of complementary words is $2^{n-1}-1$.
\end{thrm}
\begin{proof}
Let $G$ be a binary Gray code on $n-2$ bits that begins with $00\cdots0$ and ends with $11\cdots1$, either using Ruskey's construction or monotone binary Gray codes, both discussed in Section~\ref{intro}.  Applying Lemma~\ref{rev_ref} yields an $(n-1)$-bit Gray code, $H$, which begins with words $00\cdots00$, $00\cdots01$ and ends with $11\cdots11$, $11\cdots10$.  Delete the first word, $00\cdots0$ to produce $H'$.  Now construct
\[
\Gamma = 0H',1\overline{H'}
\]
\end{proof}Table~\ref{odd1} gives an example of this Gray code for $n=5$ bits.
\begin{table}[ht]
\begin{center}
\begin{tabular}{llllllll}
00001 & 00011 & 00010 & 00110 & 00111 & 00101 & 00100 & 01100 \\
& 01101 & 01001 & 01000 & 01010 & 01011 & 01111 & 01110 \\
11110 & 11100 & 11101 & 11001 & 11000 & 11010 & 11011 & 10011 \\
& 10010 & 10110 & 10111 & 10101 & 10100 & 10000 & 10001
\end{tabular}
\end{center}
\caption{A complementary Gray ordering of $\mathbb{Z}_2^5\setminus \{00000,11111\}$. \label{odd1}}
\end{table}

We can also use Lemma~\ref{rev_ref} to construct a Gray code containing all words such that the distance between complementary words is only one off half the length of the code.
\begin{thrm}
If $n$ is odd then there exists a Gray ordering of all $n$-bit binary words such that the distance between any pair of complementary words is $2^{n-1} \pm 1$.
\end{thrm}
\begin{proof}
Let $G$ be any self complementary $(n-1)$-bit Gray code.  Apply Lemma~\ref{rev_ref} to yield an $n$-bit Gray code, $H$.  Consider the word $w$ and its complement $\overline{w}$.  Their length $n-1$ prefixes are also complements and so are at distance $2^{n-2}$ in $G$.  By the construction of Lemma~\ref{rev_ref} $w$ and $\overline{w}$ are at a distance at most $2^{n-1}+1$ and at least $2^{n-1}-1$.  The parity of the words' weights establishes that complements can't be at distance $2^{n-1}$.
\end{proof}
Table~\ref{odd2} gives an example of this Gray code for $n=5$ bits.
\begin{table}[ht]
\begin{center}
\begin{tabular}{llllllll}
00000 & 00001 & 00011 & 00010 & 00110 & 00111 & 00101 & 00100 \\ 
01100 & 01101 & 01001 & 01000 & 01010 & 01011 & 01111 & 01110 \\
11110 & 11111 & 11101 & 11100 & 11000 & 11001 & 11011 & 11010 \\
10010 & 10011 & 10111 & 10110 & 10100 & 10101 & 10001 & 10000
\end{tabular}
\end{center}
\caption{A Gray ordering of $\mathbb{Z}_2^5$ with complementary words at distance $2^4\pm1 = 15$ or $17$. \label{odd2}}
\end{table}

\subsection{$q$-ary $n$-tuples for $q \geq 3$.} \label{subsec:q3}

How do we generalize the idea of complementary to $q$-ary $n$-tuples?  One often thinks of complements in $\mathbb{Z}_2^n$ as being the two words that {\em add} to $11\cdots1$.  However since $+1=-1$ in $\mathbb{Z}_2$ this can also be a subtraction.  This is the more natural notion to generalize.  In an ordering of $\mathbb{Z}_q^n$, we can ask that consecutive words be distance one in either the Hamming or the Lee metric.  We will address both metrics.

\begin{defn}
A {\em quasi-complementary} Hamming (Lee) metric Gray code of $q$-ary $n$-tuples is an ordering of the $q^n$ words such that consecutive words have Hamming (Lee) distance 1 and 
\[
G( (i + q^{n-1}) \bmod q^n) = G(i) + 111 \cdots 1 
\]
where $G(i)$ is the $i$th tuple.  Addition of tuples is done in $\mathbb{Z}_q^n$ and addition of indices is performed in $\mathbb{Z}_{q^n}$.
\end{defn}
This  definition coincides with the standard definition of complementary when $q=2$.  It ensures that all words that are separated by a multiple of $q^{n-1}$ are as far apart in the Hamming metric as possible.  Additionally, in the Lee metric, the separation in the code between words from the set is a linear function of their Lee distance.  An example of a {\em quasi-complementary} Lee metric Gray code of ternary $3$-tuples is given in Table~\ref{3^3}.
\begin{table}[ht]
\begin{center}
\begin{tabular}{lllllllll}
000 & 002 & 001 & 021 & 022 & 020 & 010 & 012 & 011 \\
111 & 110 & 112 & 102 & 100 & 101 & 121 & 120 & 122 \\
222 & 221 & 220 & 210 & 211 & 212 & 202 & 201 & 200
\end{tabular}
\end{center}
\caption{A quasi-complementary Lee metric Gray code for $\mathbb{Z}_3^3$. \label{3^3}}
\end{table}

There is a necessary condition that follows if we insist that consecutive words have distance one in the Lee metric:

\begin{prop}
If a quasi-complementary Lee metric Gray code of $q$-ary $n$-tuples exists, then either $n$ is even, $q$ is odd or $n=1$.
\end{prop}
\begin{proof}
Let $C_i$ denote the 2-regular connected graph on $i$ nodes, an $i$-cycle.  If $n\geq 3$ is odd and $q$ is even then the desired Gray code is a Hamiltonian cycle in
\[
\Box_{0 \leq i < n} C_q = C_q \Box C_q \Box \cdots \Box C_q
\]
which is a bipartite graph since it is the Cartesian product of bipartite graphs.  Since the code is quasi-complementary the words $00\cdots0$ and $11\cdots1$ are distance $q^{n-1}$ apart, which is even but are at an odd Lee distance which is a contradiction.
\end{proof}
The code for $n=1$ is trivially constructed by taking the elements of $\mathbb{Z}_q$ in the order they are generated from $1$.

Our primary construction is the higher modulus analog of standard construction of binary $n$-tuples when $n$ is even, given on Page~\ref{page:n_even}.
\begin{lemma}\label{lem:construct}
If a Hamming (Lee) metric Gray code of $q$-ary $(n-1)$-tuples exists that starts with $00\cdots0$ and ends with $11\cdots1$ then a  quasi-complementary Hamming (Lee)  metric Gray code of $q$-ary $n$-tuples exists.
\end{lemma}
\begin{proof}
Let A be the  Gray code of $q$-ary $(n-1)$-tuples satisfying the hypothesis.  Then
\[
0A+00\cdots0, 0A+11\cdots1, \ldots ,0A + (q-1)(q-1)\cdots(q-1) \] 
is the desired quasi-complementary Gray code.
\end{proof}

We now establish the existence of the necessary ingredient codes for the construction.  First when $n$ is even.

\begin{lemma} \label{lem:oddn}
If $n'$ is odd then there exists a Lee metric Gray code of $q$-ary $n'$-tuples that starts with $00\cdots0$ and ends with $11\cdots1$ for any $q \geq 3$.
\end{lemma}
\begin{proof}
We extend a monotonic binary Gray code to $q$-ary words by replacing each binary word, $w$, with a $(q-1)$-ary Gray code on the support of $w$ with symbols $\{1,2,\ldots,q-1\}$. We will replace every binary word $w$ by an instance of the $(q-1)$-ary reflected Gray code \cite{MR2251595}.  The positions with 0 in $w$ will have 0 fixed in these positions for the entire code that is replacing $w$.  The alphabet of the code replacing $w$ is $\{1,2,\ldots,q-1\}$ with sentinels 1 and $q-1$. The $(q-1)$-ary reflected Gray code never changes one sentinel directly to the other so it will remain valid for the Lee metric on $\mathbb{Z}_q$. Additionally, the  $(q-1)$-ary reflected Gray code has the property that if all the values start set to sentinels then the code will end with all positions set to sentinels.  The only decision to make is with which sentinel values the positions start.  If $w'$ is the word preceding $w$ in the binary monotone Gray code, then any position where $w$ and $w'$ have a common 1, the $(q-1)$-ary reflected Gray code replacing $w$ will be initially set at the sentinel value where the code replacing $w'$ ended in this position.  There is at most one position in $w$ that is non-zero which does not inherit its initial sentinel values from $w'$ in this way (because the Hamming distance of $w$ and $w'$ is one).  For this position, if there is one, we can choose either sentinel value because both `1' and `$q-1$' are Lee distance one from `0'.

The last word in the binary monotone Gray code is $11\cdots1$.  This is replaced by an instance of the $(q-1)$-ary reflected Gray code which will end in a word that has either a $1$ or $q-1$ in each position.  Which of these sentinels each position contains depends only and directly on the most recent choice of a sentinel value for this position to start an instance of a  $(q-1)$-ary reflected Gray code.  In order to finish with the word $11\cdots1$ we simply make the right choice for each position. 
\end{proof}
Table~\ref{qrgc} gives the Gray code for $q=4$ and $n'=3$ constructed by Lemma~\ref{lem:oddn}.  The positions where the correct choice of sentinel start positions must be made to guarantee ending in $111$ are underlined.
\begin{table}[ht]
\begin{center}
\begin{tabular}{llllllllllll}
  000 & 001 & 002 & 003 & 013 & 012 & 011 & 021 & 022 & 023 & 013 & 012 \\
  011 & 010 & 020 & 030 & \underline{1}30 & 120 & 110 & 210 & 220 & 230 & 330 & 320 \\
  310 & 300 & 200 & 100 & 10\underline{1} & 102 & 103 & 203 & 202 & 201 & 301 & 302 \\
  303 & 3\underline{3}3 & 332 & 331 & 321 & 322 & 323 & 313 & 312 & 311 & 211 & 212 \\
  213 & 223 & 222 & 221 & 231 & 232 & 233 & 133 & 132 & 131 & 121 & 122 \\
  123 & 113 & 112 & 111 
\end{tabular}
\end{center}
\caption{The Gray code for $4$-ary 3-tuples constructed by Lemma~\ref{lem:oddn}. \label{qrgc}}
\end{table}

\begin{cor} \label{cor:neven}
There exists a  quasi-complementary Lee metric Gray code of $q$-ary $n$-tuples whenever $n \in \mathbb{N}$ is even.\end{cor}

When $q$ is odd we construct an different ingredient code.
\begin{lemma}\label{lem:oddq}
If $q$ is odd then there exists a Lee metric Gray code of $q$-ary $n$-tuples that starts with $00\cdots0$ and ends with $11\cdots1$ for any $n \geq 1$.
\end{lemma}
\begin{proof}
By appealing to Lemma~\ref{lem:oddn} we may assume that $n$ is even and therefore that there exists a Gray code, $G$, of $q$-ary $(n-1)$-tuples that starts with $00\cdots0$ and ends with $11\cdots1$.  Now  
\[
0G, (q-1)G^R, (q-2)G, (q-3)G^R, \ldots, 1G
\]
is the desired Gray code for $q$-ary $n$-tuples.  The notation $G^R$ indicates the {\em reversal} of $G$: the tuples from $G$ listed in the reverse order.
\end{proof}

\begin{cor}
There exists a  quasi-complementary Lee metric Gray code of $q$-ary $n$-tuples whenever $q \geq 3$ is odd.
\end{cor}
Table~\ref{3^3} shows a quasi-complementary Lee metric Gray code for \linebreak $q=n=3$.

In the case that $n>1$ is odd and $q$ is even we show that one can construct the best possible code in the first of the senses described in Subsection~\ref{subsec:binary}, i.e. removing words from the code.  Since the Lee metric distance between $w$ and $w + 11\cdots1$ is unavoidably odd, we must remove at least one word between them (in other words, miss the vertex in the cycle on the Lee metric graph) so that the separation between them in the code matches the parity of the distance between them.  Thus, at least $q$ words must be removed in total from the code and this  lower bound on the number of missed words can be achieved.

\begin{thrm} \label{thrm:evenq1}
If $n$ is odd and $q$ is even, then there exists a  quasi-complementary Lee metric Gray code of $q$-ary $n$-tuples, missing $w+ ii\cdots i$, $0 \leq i < q$ for any choice of $w \in \mathbb{Z}_q^n$.
\end{thrm}
\begin{proof}
  We use Lemma~\ref{lem:construct} and build the ingredient in a similar manner to Lemma~\ref{lem:oddn}.  There exists a binary monotonic Gray code of $n$-tuples that starts with $100\cdots0$, $00\cdots0$ and ends at $11\cdots1$ \cite{MR1329390}.  We will expand each word to a $(q-1)$-ary reflected Gray code on its support just as in Lemma~\ref{lem:oddn}.  This will produce a $q$-ary Gray code,
  \[
    100\cdots0,\; 200\cdots0,\; \ldots,\; (q-1)00\cdots0,\; 00\cdots0,\; \ldots,\; 11\cdots1.
  \]
   We want this code to begin with $00\cdots0$ so we remove the initial tuples in pairs, $i00\cdots0$ and $(i+1)00\cdots0$, and insert them between the consecutive tuples $i00\cdots0(q-1)0\cdots0$ and $(i+1)00\cdots0(q-1)0\cdots0$ (or consecutive tuples $i00\cdots010\cdots0$ and $(i+1)00\cdots010\cdots0$) which appear in the first $(q-1)$-ary reflected Gray code placed on the first binary tuple of weight 2 that includes the first coordinate.  To ensure that the target pair of consecutive words exist consecutively in the $(q-1)$-ary reflected Gray code, we must treat the non-zero coordinate which is not the first as the ``low order'' position in the code, the one which changes most often.  

Since $q-1$ is odd the tuple $(q-1)00\cdots0$ will be left over and this is discarded, as we knew one tuple must be. In the monotonic binary Gray code between the first weight 2 tuple including the first coordinate and the last, full weight, tuple, every coordinate will be zero at least once.  This means that this slight alteration of the code does not interfere with the choice of sentinels necessary to ensure that we end with the $q$-ary tuple $11\cdots1$ in exactly the same manner at Lemma~\ref{lem:oddn}. 

We now have a code which satisfies the hypotheses of Lemma~\ref{lem:construct} and is missing $(q-1)00\cdots0$.  The Lee metric Gray code constructed will be missing the words $0(q-1)00\cdots0 + ii\cdots i$, $0 \leq i < q$.  We can choose the missing words arbitrarily to $w + ii\cdots i$ by the addition of $w-0(q-1)00\cdots0$ to every word in the sequence.

\end{proof}

The second relaxation of ``complementary'' in the binary case with an odd number of bits asks for complementary words to be separated by one more or one less than half the length of the code: $q^{n-1}$.  We show that this is not possible when $q\geq 4$ is even.  

\begin{thrm}
When $q\geq 4$ is even and $n\geq 3$ is odd, there cannot exist a Lee metric Gray code of $q$-ary $n$-tuples with the property that  words $w$ and $w + 11\cdots 1$ have separation $q^{n-1} \pm 1$.
\end{thrm}
\begin{proof}
Let such a code be $G$ and the $i$th word be $G(i)$. We will generate a contradiction.  If 
\[
G(i) + 11\cdots1 = G(i + q^{n-1} + \epsilon)
\]
then let $e(i) = \epsilon=\pm 1$.  And let $p(i)$ be the position of the word that changes between $G(i)$ and $G(i+1)$.  
If two consecutive $\epsilon$s are both positive, $e(i) = e(i+1) = 1$, then it is easy to check that $e(i+2) = 1$ since the word at separation $q^{n-1} -1$ from $G(i+2)$ is already determined to be $G(i) + 11\cdots1 \neq G(i+2) + 11\cdots1$.  And since it is not possible that all $\epsilon$s are 1, this forbids consecutive $\epsilon$s from being positive.  Similarly it can be shown that two consecutive negative $\epsilon$s are impossible.  So the $\epsilon$s must alternate between 1 and -1.

Now that we know the behaviour of $e()$, we examine the behaviour of $p()$.  Consider a place in the code where $p(i) = 1$ and  $p(i+1) \neq 1$ (one must exist because $n > 1$). Then
\[
G(i) = x w,  \;\; G(i+1) = (x\pm1) w, \;\; G(i+2) = (x \pm1) w',
\]
where $x$ is a $q$-ary digit and $w,w' \in \mathbb{Z}_q^{n-1}$.
If $e(i) = 1$, for $j=i+q^{n-1}$,  we have
\begin{align*}
&G(j) = (x\pm1+1) (w+11\cdots1), & \; &G(j+1) = (x +1) (w+11\cdots1), \\
&G(j+2) = ?, && G(j+3) = (x\pm1+1) (w' + 11\cdots1).
\end{align*}
We consider what word is $G(j+2)$?  If $p(j+1) = 1$ then $G(j+2) = (x \pm -1 +1) (w + 11\cdots1)$ which is to be distance one in the Lee metric from $G(j+3) = (x\pm1+1) (w' + 11\cdots1)$ and thus $x \pm -1 + 1 = x \pm 1 + 1$ from which we can conclude that $1=-1$ but this is not possible since $q > 2$.  Thus we have that $p(j+1) \neq 1$ and so $G(j+2) = (x+1) w''$.  Since $x+1 \neq x \pm1 +1$ we must have that $p(j+2) = 1 = p(j)$ and $w''=w'$.  This forces $p(i+2) = 1$. The same argument can be applied to $p(i+1)$ and $p(i+2)$, which differ, to conclude that $p(i+3) = p(i+1)$ and in general $p(i) = p(i \bmod 2)$ for all $i$.  Since $n > 2$ this gives a contradiction.  The proof when $e(i) = -1$ is similar.
\end{proof}

Even though a Gray code with complementary words separated by $q^{n-1}\pm 1$ does not exist, we can bound the change in separation of words that makes such a Lee metric Gray code possible.

\begin{thrm}
When $q\geq 4$ is even and $n$ is odd, there exists a Lee metric Gray code of $q$-ary $n$-tuples with the property that  words $w$ and $w + 11\cdots 1$ have separation $q^{n-1} + 1$ or $q^{n-1} - (q-1)$.
\end{thrm}
\begin{proof}
By Corollary~\ref{cor:neven} there exists a quasi-complementary Lee metric Gray code, $G$, of $q$-ary $(n-1)$-tuples.  Let $G(i)$ denote the $i$th word in $G$.  Define the code, $G'$ on $n$-tuples by

\[
G'(i) = \left \{ \begin{array}{ll} G(\lfloor i/q \rfloor) (i \bmod q)  & \mbox{ if } \lfloor i/q \rfloor \equiv 0 \bmod 2    \\
G(\lfloor i/q \rfloor) ((q-1-i) \bmod q)  & \mbox{ if } \lfloor i/q \rfloor \equiv 1 \bmod 2 \end{array}
      \right .
\]
\end{proof}

It is an unknown if the absolute value of the change (``error'') in the separation could be bounded by something smaller than $q-1$ but greater than 1.

We can also show that if we ask that consecutive words be of distance one in the Hamming metric then we can construct quasi-complementary Gray codes regardless of the parity of $n$ and $q$.  Because distance 1 in the Lee metric implies distance 1 in the Hamming metric, the only case left to prove is the case of $n$ odd and $q$ even.  Again we construct an ingredient code for use with Lemma~\ref{lem:construct}.

\begin{lemma} \label{lem:evenq3}
If $n'$ is even then there exists a Hamming metric Gray code of $q$-ary $n'$-tuples that starts with $00\cdots0$ and ends with $11\cdots1$ for any even $q \geq 4$.
\end{lemma}
\begin{proof}
By Lemma~\ref{lem:oddn} there exists a Lee metric Gray code, $G$, of $q$-ary $(n'-1)$-tuples that starts with $00\cdots0$ and ends with $11\cdots1$.  Let $G(i)$ be the $i$th word in this code.  Construct the desired code $G'$ in the following manner
\[
G'(i) = \left \{ \begin{array}{ll} G(\lfloor i/q \rfloor) (i \bmod q)  & \mbox{ if } \lfloor i/q \rfloor \equiv 0 \bmod 2    \\
G(\lfloor i/q \rfloor) ((q-1-i) \bmod q)  & \mbox{ if } \lfloor i/q \rfloor \equiv 1 \bmod 2 \end{array}
      \right .
\]
Now swapping $G'(q^n-2)$ and $G'(q^n-1)$ produces the desired code.
\end{proof}

\begin{cor}
There exists a  quasi-complementary Hamming metric Gray code of $q$-ary $n$-tuples whenever $n\geq 1$ and $q \geq 3$.
\end{cor}


\section{$n$-subsets of a $2n$-set}

The complement of an $n$-subset of a $2n$-set is itself an $n$-subset.  This allows us to sensibly define a complementary Gray code for these objects.  The weakest form of a Gray code for fixed size subsets of an $n$-set is that where the symmetric difference of consecutive sets has cardinality two; this corresponds to consecutive incidence vectors having Hamming distance two.  However, in the Gray code literature of $k$-subsets of an $n$-set, it is customary to ask for stronger ``closeness'' properties \cite{savage:97}.  The {\em strong minimal change property} asks that consecutive subsets, when written as sets in lexicographic order, be Hamming distance one; this is equivalent to changing two bit positions in the incidence vectors which must have only `0's between them.  When there are more than 2 bit positions then this property is only preserved by complementation if the two elements changing are consecutive or, equivalently, that the pair of bits that change are adjacent in the incidence vector.  This is known as the {\em adjacent interchange property}. Such Gray codes of $k$-subsets of an $n$-set exist if and only if $k = 0, 1, n$ or $n - 1$ or $n$ is
even and $k$ is odd \cite{savage:97}.  Additionally these codes cannot be cyclic because the incidence vectors $1^k0^{n-k}$ and $0^{n-k}1^k$ have degree one in the corresponding transition graph.  Since complementary Gray codes must be cyclic we cannot hope for the {\em adjacent interchange property}.

We will modify the {\em strong minimal change property} to be preserved under complementation and look for Gray codes with this new property.  
\begin{defn}
A Gray code of $k$-subsets of an  $n$-set is said to have the {\em complementary strong minimal change property} if the elements of each consecutive pair of subsets differ in the exchange of two elements such that all the elements between the two either appear in both sets or appear in neither set.  In the incidence vector representation, this is equivalent to the two bit positions which change having either all `1's or all `0's between them.
\end{defn}
We note that two incidence vectors which differ by a {\em complementary strong minimal change} have Hamming distance two and also differ by the reversal of a subword, the word between and including the two positions which change.  Subword reversal is one of the ways that Frank Ruskey presents the various stronger change properties of $k$-subsets of an $n$-set\cite{MR1288775}.  The simplicity of describing the {\em complementary strong minimal change} by subword reversals and reversals' prior use in the literature makes this new minimal change property natural.  The strongest form of the {\em complementary strong minimal change} would only permit two bits to change if they were adjacent or had a single bit position between them.  The existence of Gray codes meeting this strong notion of a complementary is an open and interesting question.

We will construct these complementary Gray codes in a manner very similar to the method of Lemma~\ref{lem:construct}.  That is, we will find a Gray code for the $n$-subsets of a $(2n-1)$-set such that the first and last sets are as complementary as possible.  We will then append a 0 to each of the incidence vectors.  The complement of the first will now be Hamming distance two from the last and so we complete the code by running through the complements of each incidence vector from the list we constructed.  
\begin{thrm} There exists a complementary Gray code for the $n$-subsets of a $2n$-set with the {\em complementary strong minimal change property}.
\end{thrm}
\begin{proof}
The Eades and McKay {\em strong minimal change property} Gray code for the $k$-subsets of an $m$-set has the property that the first incidence vector is $0^{m-k}1^k$ and the final vector is $1^k0^{m-k}$\cite{eades:84,MR1288775}.  We use this code with $k=n$ and $m=2n-1$. If we append a 0 to all vectors then the first is $0^{n-1}1^n0$ and its complement, $1^{n-1}0^n1$ is a single strong minimal change from the final vector $1^n0^n$.  So we simply finish the code with the complements of all the vectors (sets) from the first half.
\end{proof}
Table~\ref{3of6} shows a complementary Gray code for the 3-subsets of a 6-set with the complementary strong minimal change property represented as subsets.  Table~\ref{3of6a} shows the same code represented as incidence vectors.
\begin{table}[ht]
\begin{center}
\begin{tabular}{cccccccccc}
234 & 235 & 245 & 345 & 346 & 246 & 236 & 256 & 356 & 456 \\
156 & 146 & 136 & 126 & 125 & 135 & 145 & 134 & 124 & 123 \\
\end{tabular}
\caption{A complementary Gray code for the 3-subsets of a 6-set with the complementary strong minimal change property represented as subsets. \label{3of6}}
\end{center}
\end{table}

\begin{table}[ht]
\begin{center}
\begin{tabular}{cccccc}
011100  &   011010  &   010110  &   001110  &   001101  &   010101  \\
011001  &   010011  &   001011  &   000111  &   100011  &   100101  \\
101001  &   110001  &   110010  &   101010  &   100110  &   101100  \\
110100  &   111000 \\
\end{tabular}
\caption{A complementary Gray code for the 3-subsets of a 6-set with the complementary strong minimal change property represented as incidence vectors. \label{3of6a}}
\end{center}
\end{table}

\section{Permutations}

In 1974 Martin Gardner asked the question ``True or false: If $a_1a_2\ldots a_n$ is initially $12\ldots n$, Algorithm P [``Plain Changes'', or the Trotter-Johnson algorithm for an adjacent interchange Gray code of permutations \cite{MR0159764,trotter:62}] begins by visiting all $n!/2$ permutations in which 1 precedes 2; then the next permutation is $n\ldots21$.''\cite{MR2251595}.  This statement is true when $n \leq 4$ but it fails for all greater $n$ because after $n=4$, $(n-1)!/2$ is even and so the larger symbols, $n,\, n-1,\ldots$ are on the right hand side of the permutation after $n!/2$ steps, not on the left.  Seemingly implicit in this question and its false answer is the problem of determining if there exists an adjacent interchange Gray code for permutations which has the property that it begins by listing ``all $n!/2$ permutations in which 1 precedes 2; then the next permutation is $n\ldots21$.''  Such a Gray code is the most natural candidate for a complementary (really ``reverse'')  Gray code for permutations; by completing the second half of the code with the reversals of the elements from the first $n!/2$ permutations, it has the property that the $(i + n!/2)$th permutation is the reversal of the $i$th.  Reversal is a better extension of ``complement'' than the algebraic inverse because some permutations are their own inverses and this is an obstruction to the existence of an ``inverse'' Gray code.  

Considering the parity of permutations and the length of a Gray code of permutations, it is easy to generate simple necessary conditions.
\begin{prop}
If a reverse adjacent interchange Gray code for permutations of $12\ldots n$ exists for $n \geq 5$ then $n \equiv 0,1 \bmod 4$.
\end{prop}
We show that these necessary conditions are sufficient:
\begin{thrm} \label{rev_perm}
A reverse adjacent interchange Gray code for permutations of $12\ldots n$ exists if and only if $n \leq 4$ or $n \equiv 0,1 \bmod 4$.
\end{thrm}

  When $n \equiv 0 \bmod 4$, we will need our codes to have a particular property as they approach the permutation $n\ldots21$, the reversal of the first permutation, namely that the permutations just previous to this all have the symbols $1,2,\ldots, n-1$ in the order $(n-1)(n-2) \cdots 312$ and the symbol $n$ in the four left-most positions moving from right to left as shown in Table~\ref{propP}. We call this property $P$.  We note that when $n \leq 7$ the positions of $(n-1)$ will include positions to the right of the $1$ and $2$. 
\begin{table}[ht]
\begin{eqnarray*}
&& (n-1)(n-2)(n-3)n(n-4)\ldots312 \\
&& (n-1)(n-2)n(n-3)(n-4)\ldots312 \\
&& (n-1)n(n-2)(n-3)(n-4)\ldots312 \\
&& n(n-1)(n-2)(n-3)(n-4)\ldots312;
\end{eqnarray*}
\caption{Required sequence of permutations of order $n$, for Property $P$. \label{propP}}
\end{table}

\begin{thrm} \label{0mod4}
If $n \equiv 0 \bmod 4$, $n \geq 8$ and there exists a reverse Gray code for permutations of $12\ldots(n-3)$ then there exists a reverse Gray code for permutations of $12\ldots n$ with property $P$.
\end{thrm} 
\begin{proof}
We construct the desired Gray code from three different codes
\begin{itemize}
\item $G_1$, the hypothesized Gray code of order $n-3$, 
\item $G_2$, an {\em adjacent interchange property} Gray code of the $3$-subsets of an $n$ set (which exists because $n$ is even and 3 is odd \cite{savage:97}), and 
\item  $G_3$, the unique (up to reversal and starting position) cyclic Gray code for permutations of $123$, shown in Table~\ref{3perm}
\end{itemize}
\begin{table}[ht]
\begin{center}
\begin{tabular}{llllll}
123 & 213 & 231 & 321 & 312 & 132 \\
\end{tabular}
\end{center}
\caption{The cyclic Gray code of permutations of $123$. \label{3perm}}
\end{table}

We start with the permutation $123\ldots (n-3) (n-2) (n-1) n$ and apply the various Gray codes to this word in the following manner.  We first run $G_2$ on the positions of the symbols $(n-2)$, $(n-1)$ and $n$.  If these start in the right three positions they will end in the left three positions and vice versa.  In either case they will remain in the same relative order to each other.  We then apply the next transformation from $G_3$, on the all-adjacent symbols $(n-2)$, $(n-1)$ and $n$.  We run $G_3$ alternately forwards then backwards assuring that the relative positions of its symbols are either $(n-2) (n-1) n$ or $(n-2) n (n-1)$ depending on the number of iterations of $G_1$ that have passed.  At this point we run through $G_2$ again.  Once we have completed all transformations from $G_3$ with a step of $G_2$ after each one, we perform an iteration of $G_1$ on the all-adjacent symbols $12\ldots(n-3)$ and begin the iterations of $G_3$ again each followed by a step of $G_2$.

Continue in this manner until we produce the permutation
\[
(n-3) (n-4)  \ldots 3 1 2 (n-2) n (n-1).
\]
Note that this is the final iterate from $G_1$ and the odd parity of this in $G_1$ gives both the position of the largest three symbols being at the right and their order. We now run the Plain changes algorithm, moving $(n-1)$ most frequently, $n$ next frequently and $(n-2)$ least frequently, all initially towards the left.  All adjacent interchanges from now on will involve at least one of these symbols.  We run this process until we reach the permutation
\[
(n-3) (n-2) (n-4) (n-5) \ldots 3 1 2 n (n-1).
\]
At this point the only remaining interchanges permitted are $(n-3)$ with $(n-2)$ and interchanges involving one or both of $n$ and $(n-1)$.  There are essentially five permitted moves: 
\begin{itemize}
\item swap $(n-2)$ and $(n-3)$ (they must be adjacent to do this);
\item move $(n-1)$ to the left;
\item move $(n-1)$ to the right;
\item move $n$ to the left and
\item move $n$ to the right.
\end{itemize}

We now translate the problem of finishing this Gray code into finding a suitable  Hamilton path on a specific family of graphs, $\Gamma_n$.  Their vertex sets are
\[
V(\Gamma_n) = \{ (y,z,b) | \mbox{ where } b \in \{0,1\}, 1 \leq y,z \leq n \mbox{ and } y\neq z\}
\]
There is an edge between $(y,z,0)$ and $(y,z,1)$ whenever $y \neq 2$, $z \neq 2$, $(y,z) \neq (1,3), (3,1)$ or $(y,z) = (1,2), (2,1)$. There is an edge between $(y,z,b)$ and $(y\pm1,z,b)$ and $(y,z\pm1,b)$ whenever those vertices are defined.  Finally there is an edge between $(y,y+1,b)$ and $(y+1,y,b)$ whenever this pair of vertices is defined.  

The set of vertices is in bijection to the permutations (with 1 preceding 2) that remain to be added to the Gray code.  The positions of $(n-1)$ and $n$ are encoded by $y$ and $z$ respectively and $b=0$ if $(n-2)$ precedes $(n-3)$ and $b=1$ otherwise.  The edges connect any pair of vertices that differ precisely by an allowed adjacent interchange.  The restrictions on $y$ and $z$ for edges between $(y,z,0)$ and $(y,z,1)$ correspond to the fact that $(n-2)$ and $(n-3)$ cannot be swapped if $n$ or $(n-1)$ are between them in the permutation.  The permutation we have reached corresponds to the vertex $(n,n-1,1)$; for our reverse property we require to finish at $(2,1,0)$ and we must visit each vertex.

We will show that the desired Hamilton path exists for all even $n$, except $4$ (which we do not need as we already have a reverse Gray code for $n=4$).  For $n=2$ the path is 
\[
(2,1,1),\, (1,2,1),\, (1,2,0),\, (2,1,0).
\]
For $n=4$ (this will fail to be a valid Hamilton path in $\Gamma_4$ but it will be useful in the recursive construction) the path is given in Table~\ref{n4}.
\begin{table}[tb]
\begin{center}
\begin{align*}
(4,3,1),\, (4,3,0),\, (4,2,0),\, (4,2,1),\, (3,2,1),\, (2,3,1),\, (2,3,0),\, (3,2,0), \\
(3,1,0),\, (4,1,0),\, (4,1,1),\, (3,1,1),\, (2,1,1),\, (1,2,1),\, (1,3,1),\, (1,4,1), \\
(2,4,1),\, (3,4,1),\, (3,4,0),\, (2,4,0),\, (1,4,0),\, (1,3,0),\, (1,2,0),\, (2,1,0)\phantom{,} 
\end{align*}
\caption{False Hamilton path on $\Gamma_4$. \label{n4}}
\end{center}
\end{table}
To construct the Hamilton path in $\Gamma_n$ in general we will make use of the fact that an isomorphic copy of $\Gamma_{n-4}$ exists centred inside $\Gamma_n$.  We will use the Hamilton path in this $\Gamma_{n-4}$ in the recursion.  The general form of the recursion is given in Table~\ref{recursion}.
\begin{table}[tb]
  \begin{center}
  \parbox{4.2in}{
$(n,n-1,1)$,\hspace*{\fill} $(n-1,n,1)$,\hspace*{\fill} $(n-1,n,0)$,\hspace*{\fill} $(n,n-1,0)$,\hspace*{\fill} $(n,n-2,0)$, \\
$(n-1,n-2,0)$,\hspace*{\fill} $(n-1,n-3,0)$,\hspace*{\fill} $(n,n-3,0)$,\hspace*{\fill} $(n,n-3,1)$,\hspace*{\fill} $(n,n-2,1)$, \\ 
$(n-1,n-2,1)$,\hspace*{\fill} $(n-1,n-3,1)$,\hspace*{\fill} $(n-2,n-3,1)$, {\bf Follow Hamilton}\\
 {\bf path in $\Gamma_{n-4}'$ with $y'=y+2$ and $z' = z+2$. Continue: } \hspace*{\fill}$(4,3,0)$,\\
 $(4,2,0)$, $(5,2,0)$, \ldots, $(n,2,0)$, $(n,3,0)$, \ldots,  $(n,n-4,0)$,\hspace*{\fill}  $(n-1,n-4,0)$,\\
 $(n-1,n-5,0)$, \ldots, $(n-1,3,0)$, \hspace*{\fill} $(n-1,3,1)$, \hspace*{\fill} $(n-1,4,1)$, \ldots, \\
 $(n-1,n-4,1)$, \hspace*{\fill} $(n,n-4,1)$, $(n,n-5,1)$, \ldots, $(n,2,1)$,  \hspace*{\fill} $(n-1,2,1)$, \\
 \ldots, $(3,2,1)$,  \hspace*{\fill}$(2,3,1)$, \hspace*{\fill} $(2,4,1)$, \ldots, $(2,n,1)$, $(3,n,1)$, \ldots, $(n-2,n,1)$, \\
 $(n-2,n-1,1)$, $(n-3,n-1,1)$, \ldots, $(3,n-1,1)$, \hspace*{\fill}$(3,n-1,0)$,\hspace*{\fill} $(4,n-1,0)$, \\
 \ldots, $(n-2,n-1,0)$, \hspace*{\fill}$(n-2,n,0)$,  \hspace*{\fill}$(n-3,n,0)$,  \hspace*{\fill}$(n-4,n,0)$, \ldots, \\
 $(2,n,0)$, \hspace*{\fill}$(2,n-1,0)$, \ldots, $(2,3,0)$, \hspace*{\fill}$(3,2,0)$, \hspace*{\fill}$(3,1,0)$, \hspace*{\fill}$(4,1,0)$, \ldots, \\
 $(n,1,0)$,  \hspace*{\fill} $(n,1,1)$,\hspace*{\fill} $(n-1,1,1)$, \ldots, $(2,1,1)$,\hspace*{\fill} $(1,2,1)$,\hspace*{\fill} $(1,3,1)$, \ldots, \\
 $(1,n,1)$,\, $(1,n,0)$,\, $(1,n-1,0)$, \ldots, $(1,2,0)$,\, $(2,1,0)$
}
\caption{Hamilton path on $\Gamma_n$ (read left to right and top to bottom). \label{recursion}}
\end{center}
\end{table}
The only exception to this pattern is when $n=6$ for which we show the modified portion in Table~\ref{n6}.
\begin{table}[tb]
\begin{align*}
& (6,5,1),\, (5,6,1),\, (5,6,0),\, (6,5,0),\, (6,4,0),\, (6,4,1),\, (5,4,1),\, \\
& (5,4,0),\, (5,3,0),\, (5,3,1),\, (4,3,1),\, \mbox{\bf Follow Hamilton path } \\
&  \mbox{\bf in $\Gamma_{2}'$ with $y'=y+2$ and $z'=z+2$. Continue: } (4,3,0),\, \\
& (4,2,0),\, (5,2,0), (6,2,0),\, (6,3,0),\, (6,3,1),\, (6,2,1),\, (5,2,1),   \\
& (4,2,1),\, (3,2,1),\, (2,3,1),\, (2,4,1), \ldots, (2,n,1),\, (3,n,1), \ldots, \\
& \mbox{\bf The rest is the same as the general recursion for $\Gamma_{n}$}.
\end{align*}
\caption{Modified portion of Hamilton path on $\Gamma_6$ (read left to right and top to bottom). \label{n6}}
\end{table}

The code so far contains all permutations in which $1$ precedes $2$, begins with $12\ldots n$ and ends with $n\ldots312$ which is adjacent to $n\ldots321$, the reversal of the first permutation.  We now complete the code by appending the reversals of the first $n!/2$ permutations.  We note that the structure of the recursion on the $\Gamma_n$ produces a code with property $P$.
\end{proof}
Theorem~\ref{0mod4} uses the existence of codes for $n \equiv 1 \bmod 4$ to construct the code of order $n+3$, the next largest $n \equiv 0 \bmod 4$. We now give Theorem~\ref{1mod4} which completes the induction; it is a construction for proceeding from $n$ to $n+1$ when $n \equiv 0 \bmod 4$.
\begin{thrm} \label{1mod4}
If $n \equiv 1 \bmod 4$, $n \geq 5$ and there exists a reverse Gray code for permutations of $12\ldots(n-1)$ with property $P$, then there exists a reverse Gray code for permutations of $12\ldots n$.
\end{thrm}
\begin{proof}
We run a modification of the ``Plain changes'' algorithm \cite{MR2251595}.  We take the posited reverse Gray code, $G$,  for permutations of $12\ldots (n-1)$, replicate each permutation $n$ times and insert the symbol $n$ in all possible $n$ positions starting with $12\ldots (n-1)n$ and sweeping $n$ from right to left, then left to right alternately, as occurs in the Trotter-Johnson Algorithm \cite{MR0159764,trotter:62}.  Because $G$ has property $P$ we will eventually reach the word
\[
n (n-2) (n-3) (n-4) (n-1) (n-5) \ldots 3 1 2
\]
At this point we finish the code as shown in Table~\ref{tab:1mod4}. 
\begin{table}[ht]
\[
\begin{array}{cccccccccc}
n & (n-2) & (n-3) & (n-4) & (n-1)  & (n-5) & \ldots & 3 & 1 & 2 \\
n & (n-2) & (n-3) & (n-1) & (n-4)  & (n-5) & \ldots & 3 & 1 & 2 \\
n & (n-2) & (n-1) & (n-3) & (n-4)  & (n-5) & \ldots & 3 & 1 & 2 \\
(n-2) & n &(n-1) & (n-3) & (n-4)  & (n-5) & \ldots & 3 & 1 & 2 \\
(n-2) & n &(n-3) & (n-1) & (n-4)  & (n-5) & \ldots & 3 & 1 & 2 \\
(n-2) & (n-3) & n & (n-1) & (n-4)  & (n-5) & \ldots & 3 & 1 & 2 \\
(n-2) & (n-3) & (n-1) & n & (n-4)  & (n-5) & \ldots & 3 & 1 & 2 \\
(n-2) & (n-3) & (n-1) & (n-4)  & n & (n-5) & \ldots & 3 & 1 & 2 \\
(n-2) & (n-3) & (n-1) & (n-4)  & (n-5) & n & \ldots & 3 & 1 & 2 \\
&&& \vdots &&&&&&\\
(n-2) & (n-3) & (n-1) & (n-4)  & (n-5) & \ldots & 3 & 1 & 2 & n\\
(n-2) & (n-1) & (n-3) & (n-4)  & (n-5) & \ldots & 3 & 1 & 2 & n\\
(n-1) & (n-2) & (n-3) & (n-4)  & (n-5) & \ldots & 3 & 1 & 2 & n\\
(n-1) & (n-2) & (n-3) & (n-4)  & (n-5) & \ldots & 3 & 1 & n & 2\\
(n-2) & (n-1) & (n-3) & (n-4)  & (n-5) & \ldots & 3 & 1 & n & 2\\
(n-2) & (n-1) & (n-3) & (n-4)  & (n-5) & \ldots & 3 & n & 1 & 2\\
(n-1) & (n-2) & (n-3) & (n-4)  & (n-5) & \ldots & 3 & n & 1 & 2\\
(n-1) & (n-2) & (n-3) & (n-4)  & (n-5) & \ldots & n & 3 & 1& 2\\
(n-2) & (n-1) & (n-3) & (n-4)  & (n-5) & \ldots & n & 3 & 1 & 2\\
&&& \vdots &&&&&&\\
(n-1) & (n-2) & (n-3) & n & (n-4)  & (n-5) & \ldots & 3 & 1 & 2 \\
(n-2) & (n-1) & (n-3) & n & (n-4)  & (n-5) & \ldots & 3 & 1 & 2 \\
(n-2) & (n-1) & n & (n-3) & (n-4)  & (n-5) & \ldots & 3 & 1 & 2 \\
(n-1) & (n-2) & n & (n-3) & (n-4)  & (n-5) & \ldots & 3 & 1 & 2 \\
(n-1) & n & (n-2) & (n-3) & (n-4)  & (n-5) & \ldots & 3 & 1 & 2 \\
n & (n-1) & (n-2) & (n-3) & (n-4)  & (n-5) & \ldots & 3 & 1 & 2 \\
\end{array}
\]
\caption{The end of the first half of the code in the case $n \equiv 1 \bmod 4$. \label{tab:1mod4}}
\end{table}
Then swap the $1$ and the $2$ to produce the reversal of the very first permutation and complete the code with the reversal of the first $n!/2$ permutations.
\end{proof}
Since the Trotter-Johnson algorithm produces reverse Gray codes for $n \leq 4$, the existence of all the desired Gray codes is established and Theorem~\ref{rev_perm} is a corollary of Theorems~\ref{0mod4} and~\ref{1mod4}.
Table~\ref{n5} displays a full reverse Gray code for permutations of order 5.
\begin{table}[ht]
\begin{center}
\begin{tabular}{llllllll}
  12345 & 12354 & 12534 & 15234 & 51234 & 51243 & 15243 & 12543 \\
  12453 & 12435 & 14235 & 14253 & 14523 & 15423 & 51423 & 54123 \\
  45123 & 41523 & 41253 & 41235 & 41325 & 41352 & 41532 & 45132 \\
  54132 & 51432 & 15432 & 14532 & 14352 & 14325 & 13425 & 13452 \\
  13542 & 15342 & 51342 & 51324 & 15324 & 13524 & 13254 & 13245 \\
  31245 & 31254 & 31524 & 35124 & 53124 & \multicolumn{3}{l}{\bf \small Now we stop running} \\
\multicolumn{8}{l}{\bf \small Plain Changes and start the Theorem~\ref{1mod4} modified order:}\\
  53142 & 53412 & 35412 & 35142 & 31542 & 31452 & 31425 & 34125 \\
  43125 & 43152 & 34152 & 34512 & 43512 & 45312 & 54312 &  \\
\multicolumn{8}{p{4.2in}}{\bf \small We are halfway. We finish with the permutations from the} \\
  \multicolumn{5}{l}{\bf \small first half in order, each reversed:}&          54321 & 45321 & 43521\\
  43251 & 43215 & 34215 & 34251 & 34521 & 35421 & 53421 & 53241 \\
  35241 & 32541 & 32451 & 32415 & 32145 & 32154 & 32514 & 35214 \\
  53214 & 52314 & 25314 & 23514 & 23154 & 23145 & 23415 & 23451 \\
  23541 & 25341 & 52341 & 52431 & 25431 & 24531 & 24351 & 24315 \\
  42315 & 42351 & 42531 & 45231 & 54231 & 54213 & 45213 & 42513 \\
  42153 & 42135 & 24135 & 21435 & 21453 & 24153 & 24513 & 25413 \\
  52413 & 52143 & 52134 & 25134 & 25143 & 21543 & 21534 & 21354 \\
  21345
\end{tabular}
\caption{A reverse Gray code for the permutations of order 5. \label{n5}}
\end{center}
\end{table}

We finish with a related but slightly off topic extension of the investigation of the graph, $\Gamma_n$, which was introduced in the proof of Theorem~\ref{0mod4}.  When $n$ is odd and we only take the $b=0$ layer, a Hamilton cycle in $\Gamma_n$ is a Gray code for all permutations of the multiset $\{1,2,3,3,\ldots,3\}$.  The only known Gray codes we could find for this combinatorial family were not cyclic \cite{MR1161985,MR0172816} so we feel it worth discussing.  For $n=3$ the Hamilton cycle is 
\[
(1,3),\, (2,3),\, (3,2),\, (3,1),\, (2,1),\, (1,2),\, 
\] 
For larger $n$ we use the recursion
\begin{align*}
& (1,3),\, (2,3),\; \mbox{\bf Follow the Hamilton cycle in $\Gamma_{n-2}'$ with } \\
&  \mbox{\bf $y'=y+1$ and $z'=z+1$. Continue: } (2,4),\, (1,4),\, (1,5), \\ 
& (1,6), \ldots, (1,n),\, (2,n), \ldots, (n-1,n),\, (n,n-1),\, (n,n-2), \\
& (n,n-3), \ldots, (n,1),\, (n-1,1), \ldots, (2,1),\, (1,2)
\end{align*}

An interesting  open question in reverse Gray codes for permutations is whether weakened reverse codes can be found for $n \equiv 2,3 \bmod 4$ in either of the two senses discussed in Section~~\ref{sec:qn}.

\def\cprime{$'$}


\end{document}